\documentclass[a4paper,10pt]{amsart}

\usepackage{amsmath}
\usepackage{amsthm}
\usepackage{amssymb}
\usepackage{amsfonts}
\usepackage{mathrsfs}  
\usepackage{enumitem} 
\usepackage[bookmarks=false,hyperindex,pdftex,colorlinks,citecolor=blue,urlcolor=cyan]{hyperref}
\usepackage[a4paper,lmargin=3cm,rmargin=3cm,tmargin=4cm,bmargin=4cm,marginparwidth=2.8cm,marginparsep=1mm]{geometry} 

\DeclareMathOperator{\lspan}{span}                          
\DeclareMathOperator{\conv}{conv}                           
\DeclareMathOperator{\Lip}{Lip}                             

\newcommand{\RR}{\mathbb{R}}                                
\newcommand{\ep}{\varepsilon}

\newcommand{\abs}[1]{\left|{#1}\right|}                     
\newcommand{\pare}[1]{\left({#1}\right)}                    
\newcommand{\set}[1]{\left\{{#1}\right\}}                   
\newcommand{\norm}[1]{\left\|{#1}\right\|}                  
\newcommand{\duality}[1]{\left<{#1}\right>}                 
\newcommand{\cl}[1]{\overline{#1}}                          
\newcommand{\wscl}[1]{\overline{#1}^{w^*}}                  
\newcommand{\wconv}{\stackrel{w}{\longrightarrow}}          
\newcommand{\restrict}{\mathord{\upharpoonright}}           

\newcommand{\lipfree}[1]{\mathcal{F}({#1})}                 
\newcommand{\lipnorm}[1]{\norm{#1}_L}                       
\newcommand{\ideal}[1]{\mathcal{I}({#1})}
\newcommand{\const}[1]{\mathcal{K}({#1})}

\renewcommand{\leq}{\leqslant}
\renewcommand{\geq}{\geqslant}

\theoremstyle{plain}
\newtheorem{theorem}{Theorem}
\newtheorem{lemma}[theorem]{Lemma}

\newtheorem{proposition}[theorem]{Proposition}

\newtheorem*{claim*}{Claim}

\theoremstyle{definition}
\newtheorem*{definition*}{Definition}
\newtheorem{definition}[theorem]{Definition}

\newtheorem{remark}[theorem]{Remark}

\begin{document}
\title{The unique predual problem for Lipschitz spaces, revisited}

\author[R. J. Aliaga]{Ram\'on J. Aliaga}
\address[R. J. Aliaga]{Instituto Universitario de Matem\'atica Pura y Aplicada,
Universitat Polit\`ecnica de Val\`encia,
Camino de Vera S/N,
46022 Valencia, Spain}
\email{ramon.aliaga@upv.es}

\author[F. Vico]{Felipe Vico}
\address[F. Vico]{Instituto de Telecomunicaciones y Aplicaciones Multimedia,
Universitat Polit\`ecnica de Val\`encia,
Camino de Vera S/N,
46022 Valencia, Spain}
\email{fevibon@teleco.upv.es}

\begin{abstract}
In his 2018 paper ``On the unique predual problem for Lipschitz spaces'', N. Weaver published proofs that Banach spaces $\mathrm{Lip}_0(M)$ of Lipschitz functions on a complete metric space $M$ have strongly unique preduals whenever $M$ has finite diameter or is geodesic. A gap was recently noticed in the proof of a crucial lemma that claimed that the property of having a strongly unique predual passes to $1$-codimensional weak$^*$-closed subspaces. In this note, we confirm that the lemma is actually false by providing an explicit counterexample. We also expand on some of Weaver's original arguments to provide a new, valid proof of the following particular case: $\mathrm{Lip}_0(M)$ has a strongly unique predual whenever $M$ is a convex subset of a finite-dimensional normed space.
\end{abstract}



\maketitle

\section*{Introduction}

Let $X$ be a dual Banach space. We say that $X$ has a \emph{unique predual} if there is a unique, up to linear isometry, Banach space $Y$ such that $Y^*$ is linearly isometric to $X$, and we say that $X$ has a \emph{strongly unique predual} if any linear isometry from $X$ onto a dual Banach space $Y^*$ is weak$^*$-to-weak$^*$ continuous, hence the adjoint of an isometry between the corresponding preduals. We may identify preduals of $X$ as subspaces of the dual $X^*$ via the canonical inclusion of every Banach space in its bidual; explicitly, a closed subspace $Y\subset X^*$ is a predual of $X$ if the mapping $T:X\to Y^*$ given by $\duality{Tx,y}=\duality{y,x}$ for $x\in X$, $y\in Y$ is a surjective linear isometry. Using this identification, $X$ has a strongly unique predual if there is exactly one such subspace $Y\subset X^*$, and it has a unique predual if all such $Y$ are isometric. The latter is a formally weaker property, and it is currently unknown whether both properties are actually equivalent. Examples of spaces that are strongly unique preduals (i.e. such that their duals have strongly unique preduals) include spaces with the Radon-Nikod\'ym property \cite{Godefroy_predual}; preduals of von Neumann algebras, including the usual $L_1$-spaces \cite{Godefroy_predual}; and separable L-embedded spaces \cite{Pfitzner_x}. The sequence space $\ell_1$ is maybe the simplest example of a dual Banach space failing to have a unique predual.


In \cite{Weaver18}, N. Weaver studied the unique predual problem for Lipschitz spaces, i.e. Banach spaces of real-valued Lipschitz functions on a metric space $M$. These spaces will be defined formally in Section \ref{sec:theorem 2}; for now, we will just mention that there are two possible variants of such spaces, denoted $\Lip(M)$ and $\Lip_0(M)$. The theory of Lipschitz spaces bears some resemblances with the theory of von Neumann algebras \cite{AP22,Weaver2,Weaver96_3}, and so one could expect them to have strongly unique preduals as well. Weaver confirmed this much for $\Lip(M)$ spaces \cite[Theorem 2.4]{Weaver18}, with a proof reminiscent of Sakai's original proof for von Neumann algebras \cite{Sakai}. He then deduced that spaces of the form $\Lip_0(M)$ have strongly unique preduals as well when $M$ has finite diameter, by means of the following lemma attributed to U. Bader.

\setcounter{theorem}{-1} 
\begin{lemma}[{\cite[Lemma 3.1]{Weaver18}}]
\label{lm:weaver lemma}
Let $X$ be a dual Banach space with a strongly unique predual and let $V$ be a weak$^*$ closed subspace of $X$ of codimension $1$. Then $V$ has a strongly unique predual.
\end{lemma}

\noindent The connection to Lemma \ref{lm:weaver lemma} is that, for bounded $M$, $\Lip_0(M)$ is a weak$^*$ closed, $1$-codimensional subspace of some $\Lip(M')$. Based on the bounded case, Weaver was further able to deduce that $\Lip_0(M)$ has a strongly unique predual whenever $M$ is geodesic, including the important cases where $M$ is a Banach space with the norm metric or a convex subset thereof.

Unfortunately, in 2026 M. Gonz\'alez reported a gap in the proof of Lemma \ref{lm:weaver lemma} that called all results for $\Lip_0$ spaces into question (see \cite{Weaver_corrigendum} and the updated arXiv preprint \cite{Weaver_arxiv}, where the issue is acknowledged). It remained unclear whether the statement of Lemma \ref{lm:weaver lemma} itself was correct and its proof could be amended. In this note, we settle this question and show that Lemma \ref{lm:weaver lemma} is, in fact, false by constructing an explicit counterexample.

\begin{theorem}
\label{th:theorem 1}
There exists a Banach space $X$ such that
\begin{enumerate}[label={\upshape{(\alph*)}}]
\item $X$ is isomorphic to $\ell_1$,
\item $X$ has a strongly unique predual, and
\item $X$ has a weak$^*$ closed, $1$-complemented, $1$-codimensional subspace isometric to $\ell_1$.
\end{enumerate}
\end{theorem}

While Lemma \ref{lm:weaver lemma} is false, the rest of Weaver's arguments in \cite{Weaver18} remain valid to the best of our knowledge, including the passage from ``$\Lip_0(M)$ has a strongly unique predual when $M$ is bounded'' to ``$\Lip_0(M)$ has a strongly unique predual when $M$ is geodesic'' as it is independent of Lemma \ref{lm:weaver lemma} \cite{Weaver_arxiv,Weaver_corrigendum}. Expanding on the proof of that implication, we are able to provide a new, valid proof of the following particular case.

\begin{theorem}
\label{th:theorem 2}
If $M$ is a convex subset of a finite-dimensional Banach space then $\Lip_0(M)$ has a strongly unique predual.
\end{theorem}

Our proof of Theorem \ref{th:theorem 2} combines arguments from \cite{Weaver18} with a new induction process on the dimension of $M$ so, instead of reducing the problem to the bounded case, we reduce it to the one-dimensional case where $\Lip_0(M)=L_\infty$ is known to have a strongly unique predual. As a particular case of Theorem \ref{th:theorem 2}, Lipschitz spaces over finite-dimensional Banach spaces have strongly unique preduals. Whether $\Lip_0(M)$ has a unique predual when $M$ is an infinite-dimensional Banach space (or \textit{any} metric space, for that matter) remains an open problem.

The proofs of Theorems \ref{th:theorem 1} and \ref{th:theorem 2} are provided in the next two independent sections. For the remainder of the paper, $B_X$ will denote the closed unit ball of a Banach space $X$ as usual.

\section{Proof of Theorem 1}
\label{sec:theorem 1}

We start by defining the predual of $X$ in Theorem \ref{th:theorem 1}. It will be an equivalent renorming of the sequence space $c_0$ with real scalars. We denote by $e_n$ the $n$-th canonical basis vector of $c_0$, and set
$$
\Gamma = \set{e_1+e_n,e_1-e_n,-e_1+e_n,-e_1-e_n \,:\, n\geq 2} \subset c_0 .
$$
Then we let $\norm{\cdot}_Y$ be the norm on $c_0$ whose unit ball is the set
$$
B_Y = \cl{\conv}\pare{\Gamma \cup \set{y\in B_{c_0} : y_1=0}} ,
$$
and put $Y=(c_0,\norm{\cdot}_Y)$. To see that this does in fact define an equivalent renorming of $c_0$, it suffices to verify that
\begin{equation}
\label{eq:Y eqv norm}
\tfrac{1}{2}B_{c_0} \subset B_Y \subset B_{c_0} .
\end{equation}
The rightmost inclusion in \eqref{eq:Y eqv norm} is clear as $\Gamma\subset B_{c_0}$. For the other one, note that
$$
e_1 = \tfrac{1}{2}(e_1+e_2) + \tfrac{1}{2}(e_1-e_2) \in \conv(\Gamma)
$$
and similarly $-e_1\in\conv(\Gamma)$, hence $te_1\in\conv(\Gamma)$ for all $t\in [-1,1]$. Thus, for any $y\in B_{c_0}$,
$$
\tfrac{1}{2}y = \tfrac{1}{2}y_1e_1 + \tfrac{1}{2}(y-y_1e_1) \in B_Y .
$$

Now let $X=Y^*$. Then $X$ is an equivalent renorming of $\ell_1$. In fact, we may describe its norm $\norm{\cdot}_X$ explicitly: for any $x\in\ell_1$, we claim that
\begin{equation}
\label{eq:X eqv norm}
\norm{x}_X = \max\set{ \sum_{n\geq 2}\abs{x_n} , \abs{x_1}+\max_{n\geq 2}\abs{x_n} } .
\end{equation}
Indeed, note that
$$
\norm{x}_X = \sup_{y\in B_Y}\duality{x,y} = \sup_{y\in B_Y}\sum_{n=1}^\infty x_ny_n = \sup\set{ \sum_{n=1}^\infty x_ny_n : y\in\Gamma \text{ or } y\in B_{c_0},y_1=0 } .
$$
If $y\in B_{c_0}$ with $y_1=0$, the sum is bounded by $\sum_{n\geq 2}\abs{x_n}$ and this value can be approached by taking $y_n=\mathrm{sign}(x_n)$ for $2\leq n\leq N$ and $y_n=0$ for $n>N$. Similarly, for $y\in\set{e_1\pm e_n,-e_1\pm e_n}$, $n\geq 2$ the maximum value of $\duality{x,y}$ is $\abs{x_1}+\abs{x_n}$, and taking the maximum over $n$ yields \eqref{eq:X eqv norm}.

Let $V$ be the subspace of $X$ given by
$$
V = \set{x\in X : x_1=0} .
$$
It is clear that $V$ has codimension $1$ in $X$, and moreover $V=\ker(e_1)$ with $e_1\in Y$ so it is weak$^*$ closed. Applying \eqref{eq:X eqv norm} we see that $\norm{x}_X=\norm{x}_1$ for any $x\in V$, so $V$ is isometric to $\ell_1$. Finally, note that the mapping $P:X\to V$ given by $Px=x-x_1e_1^*$ (where $e_n^*$ denotes the $n$-th canonical basis vector of $\ell_1$) is a linear projection onto $V$, and $\norm{Px}_X = \norm{Px}_1 = \sum_{n\geq 2}\abs{x_n} \leq \norm{x}_X$, so $\norm{P}=1$. Thus $V$ is $1$-complemented. This shows that $X$ and $V$ satisfy items (a) and (c) in Theorem \ref{th:theorem 1}.

In order to establish condition (b), we will use the following criterion for uniqueness of the predual. It can be essentially found in Godefroy's survey \cite{Godefroy_predual}.

\begin{proposition}[cf.~\cite{Godefroy_predual}]
\label{pr:godefroy criterion}
Let $X$ be a dual Banach space, and let $\mathscr{C}_{w^*,w}$ denote the set of points of continuity of the identity mapping $(B_{X^*},w^*)\to(B_{X^*},w)$. Suppose that $\cl{\lspan}(\mathscr{C}_{w^*,w})$ is a norming subspace of $X^*$. Then $X$ has a strongly unique predual.
\end{proposition}

\noindent The statement of Proposition \ref{pr:godefroy criterion} is given in \cite[Example II.2.1]{Godefroy_predual} with the condition ``$\cl{\lspan}(\mathscr{C}_{w^*,w})$ is norming'' replaced by the stronger ``$\wscl{\conv}(\mathscr{C}_{w^*,w})=B_{X^*}$'', but the same argument (based on combining Lemma I.5, Lemma I.2 and Theorem II.1 in \cite{Godefroy_predual}) yields our slightly more general formulation. A short, self-contained proof is also provided below.

We will prove that $\Gamma\subset\mathscr{C}_{w^*,w}$. This will be enough to finish the proof of Theorem \ref{th:theorem 1}, as
$$
e_n = \tfrac{1}{2}(e_1+e_n) + \tfrac{1}{2}(-e_1+e_n) \in \conv(\Gamma)
$$
for all $n\geq 2$ and we already saw that $e_1\in\conv(\Gamma)$ as well, thus $\cl{\lspan}(\Gamma)=Y$ is norming for $X$. Note that $\cl{\conv}(\Gamma)\neq B_Y$, hence why we need the stronger formulation of Proposition \ref{pr:godefroy criterion}.

Let us see that $y_n=e_1+e_n$ belongs to $\mathscr{C}_{w^*,w}$. Set $x_n=\frac{1}{3}e_1^*+\frac{2}{3}e_n^*\in B_X$. It is clear that $\duality{x_n,y_n}=1$. We also have $\duality{x_n,y}\leq\frac{1}{3}$ for any $y\in\Gamma\setminus\set{y_n}$ and $\duality{x_n,y}\leq\frac{2}{3}$ for any $y\in B_{c_0}$ with $y_1=0$. It follows that $\duality{x_n,c}\leq\frac{2}{3}$ for any $c$ belonging to the set
$$
C = \wscl{\conv}\big( \Gamma\setminus\set{y_n} \cup \set{y\in B_{c_0} : y_1=0} \big) \subset B_{X^*} .
$$
Note that $B_{X^*}=\conv(C\cup\set{y_n})$. Indeed, $C$ is weak$^*$ compact, therefore so is $\conv(C\cup\set{y_n})$, and the latter contains $\Gamma$ and all $y\in B_{c_0}$, $y_1=0$, so it contains $B_Y$ and thus $B_{X^*}$ by Goldstine's theorem. Therefore, any element $z\in B_{X^*}$ can be written as $z=ty_n+(1-t)c$ for some $t\in [0,1]$ and $c\in C$. For such an expression, we have
$$
\duality{x_n,z} = t\duality{x_n,y_n} + (1-t)\duality{x_n,c} \leq t+\tfrac{2}{3}(1-t) = 1-\tfrac{1}{3}(1-t)
$$
hence
$$
1-t \leq 3(1-\duality{x_n,z}) .
$$
It follows that
\begin{equation}
\label{eq:exp norm bound}
\norm{z-y_n}_{X^*} = (1-t)\norm{y_n-c}_{X^*} \leq 2(1-t) \leq 6(1-\duality{x_n,z}) .
\end{equation}
We conclude that $y_n$ is a point of weak$^*$-to-weak, even weak$^*$-to-norm continuity in $B_{X^*}$. Indeed, if $(z_i)$ is a net in $B_{X^*}$ that converges weak$^*$ to $y_n$ then in particular $\duality{x_n,z_i}$ converges to $\duality{x_n,y_n}=1$ and therefore $\norm{z_i-y_n}_{X^*}$ converges to $0$ by \eqref{eq:exp norm bound}.

A similar argument shows that $y'_n=e_1-e_n$ is a point of weak$^*$-to-norm continuity, by evaluating against $x'_n=\frac{1}{3}e_1^*-\frac{2}{3}e_n^*\in B_X$. By symmetry, this finishes the proof that $\Gamma\subset\mathscr{C}_{w^*,w}$ and thus of Theorem \ref{th:theorem 1}.

\begin{remark}
The space $X$ is very far from resembling a Lipschitz space (for instance, infinite-dimensional Lipschitz spaces cannot be separable because they contain $\ell_\infty$ \cite{CuthJohanis}), so Theorem \ref{th:theorem 1} does not preclude the possibility that Lemma \ref{lm:weaver lemma} may hold for Lipschitz spaces in particular.
\end{remark}

We end this section by providing a short proof of Proposition \ref{pr:godefroy criterion}.

\begin{proof}[Proof of Proposition \ref{pr:godefroy criterion}]
We will show that the only predual of $X$, seen as a subspace of $X^*$, is the space $N=\cl{\lspan}(\mathscr{C}_{w^*,w})$. Indeed, fix a predual $Y\subset X^*$ of $X$, and let us see that $Y$ contains $\mathscr{C}_{w^*,w}$. Let $c\in\mathscr{C}_{w^*,w}\subset B_{X^*}$, then by Goldstine's theorem there exists a net $(y_i)$ in $B_Y$ that converges weak$^*$ to $c$. Since $c\in\mathscr{C}_{w^*,w}$ we actually have $y_i\wconv c$ and therefore $c\in Y$. We conclude that $N\subset Y$. On the other hand, if there exists $y\in Y\setminus N$ then the Hahn-Banach theorem yields an element of $B_{Y^*}$, which we identify with $B_X$, that vanishes on $N$. Thus $N$ cannot be norming for $X$, a contradiction. Hence $Y=N$ as required.
\end{proof}

\section{Proof of Theorem 2}
\label{sec:theorem 2}

We start this section by defining the Lipschitz spaces under study. Let $M$ be a metric space with metric $d$. We assume tacitly that $M$ is a \emph{pointed} metric space, meaning that we have chosen a distinguished point $0\in M$ as a base point for our definitions. Then we define the \emph{Lipschitz space} over $M$ as
$$
\Lip_0(M) = \set{f:M\to\RR \,:\, \text{$f$ is Lipschitz and $f(0)=0$}} .
$$
This becomes a Banach space when endowed with the Lipschitz norm given by
$$
\lipnorm{f} = \sup\set{\frac{f(x)-f(y)}{d(x,y)} \,:\, x\neq y\in M}
$$
(note that, if the restriction $f(0)=0$ is lifted, $\lipnorm{\cdot}$ is merely a seminorm). If a different base point $0'\in M$ is chosen, the resulting Lipschitz space $\Lip_{0'}(M)$ is linearly isometric to $\Lip_0(M)$ via the mapping $f\mapsto f-f(0')$. Since every Lipschitz function on $M$ can be uniquely extended to the closure of $M$ with the same Lipschitz constant, the Lipschitz space over $M$ can also be identified with that over the completion of $M$. Thus, we will usually assume that $M$ is complete. In particular, in Theorem \ref{th:theorem 2} we may assume that $M$ is closed as the closure of a convex set is convex.

Lipschitz spaces are dual Banach spaces. The simplest way to construct their canonical preduals is as follows. Consider the evaluation functionals $\delta(x)\in\Lip_0(M)^*$, $x\in M$, given by $f\mapsto f(x)$. These functionals generate a subspace
$$
\lipfree{M} = \cl{\lspan}\set{\delta(x)\,:\,x\in M} ,
$$
of $\Lip_0(M)^*$, called the \emph{Lipschitz-free space} over $M$. The space $\lipfree{M}$ is always an isometric predual of $\Lip_0(M)$ under the usual identification; hence, if $\Lip_0(M)$ has a strongly unique predual, then that predual must be $\lipfree{M}$. The weak$^*$ topology induced by $\lipfree{M}$ on $\Lip_0(M)$ agrees on bounded subsets of $\Lip_0(M)$ (but not on the Lipschitz space as a whole) with the topology of pointwise convergence. For further reference on Lipschitz and Lipschitz-free spaces, we direct the reader to the monograph \cite{Weaver2} (where $\lipfree{M}$ is denoted $\text{\AE}(M)$ and called ``Arens-Eells space'' instead).

Let $N$ be a subset of $M$ containing $0$. By McShane's theorem \cite[Theorem 1.33]{Weaver2}, every function in $\Lip_0(N)$ can be extended to a function in $\Lip_0(M)$ without increasing its Lipschitz constant. As a consequence, $\lipfree{N}$ can be identified isometrically with the subspace
$$
\mathcal{F}_M(N) = \cl{\lspan}\set{\delta(x) \,:\, x\in N}
$$
of $\lipfree{M}$ (see \cite[Theorem 3.7]{Weaver2}), and we will usually do so by omitting the subscript $M$ whenever it is clear from context. We also consider the following subspaces of $\Lip_0(M)$:
\begin{align*}
\ideal{N} &= \set{ f\in\Lip_0(M) \,:\, \text{$f(x)=0$ for all $x\in N$} } \\
\const{N} &= \set{ f\in\Lip_0(M) \,:\, \text{$f$ is constant on $N$}}
\end{align*}
To be more precise we should write $\mathcal{I}_M(N)$ and $\mathcal{K}_M(N)$, but again the ambient metric space $M$ will usually be fixed and so we will omit it for simplicity. Note that $\const{N}=\ideal{N}$ if the base point belongs to $N$ or to its closure, but both spaces may be different in general.

\bigskip

Recall that a metric space $M$ is \emph{geodesic} if, for any $x,y\in M$ with $r=d(x,y)$, one can find an isometric copy of the interval $[0,r]\subset\RR$ in $M$ whose endpoints are $x$ and $y$. In particular, given any $t\in [0,r]$, one can find a point $z\in M$ such that $d(x,z)=t$ and $d(y,z)=r-t$ (and therefore $d(x,z)+d(z,y)=d(x,y)$).

In \cite{Weaver18}, Weaver gave a proof that $\Lip_0(M)$ has a strongly unique predual whenever $M$ has finite diameter based on the faulty Lemma \ref{lm:weaver lemma}, and then showed that this implies the same property whenever $M$ is geodesic. While the full result is no longer valid, the implication ``bounded $\Rightarrow$ geodesic'' is still correct as it is independent of Lemma \ref{lm:weaver lemma} (see \cite[Theorem 3.1]{Weaver_arxiv} or \cite[Theorem 3.27]{Weaver2}). The proof of this implication has the two following fundamental steps, presented here in opposite order as in the original:
\begin{itemize}
\item If $M$ is geodesic then, for any ball $B$ with center $0$, $\ideal{B}$ is weak$^*$ closed with respect to any predual of $\Lip_0(M)$.
\item If moreover $\Lip_0(B)$ has a strongly unique predual, then any predual of $\Lip_0(M)$ must contain $\lipfree{B}$.
\end{itemize}
Thus, if all $\Lip_0(B)$ have a strongly unique predual, then every predual of $\Lip_0(M)$ must contain $\lipfree{B}$ for all $B$, hence $\lipfree{M}$ in its entirety.

In the second step above, preduals of $\Lip_0(M)$ are again identified with closed subspaces of $\Lip_0(M)^*$. Since Weaver provides only a few details of that step in \cite{Weaver18}, we now state it in full rigor and provide details of the proof, for completeness and for future reference.


\begin{lemma}
\label{lm:predual and ideal}
Let $Y\subset\Lip_0(M)^*$ be a predual of $\Lip_0(M)$. Let $N\subset M$ be such that $0\in N$ and suppose that
\begin{enumerate}[label={\upshape{(\roman*)}}]
\item $\ideal{N}$ is closed in the weak topology $\sigma(\Lip_0(M),Y)$ induced by $Y$, and
\item $\lipfree{N}$ is the strongly unique predual of $\Lip_0(N)$.
\end{enumerate}
Then $\mathcal{F}_M(N)\subset Y$.
\end{lemma}

\begin{proof}
Let $Y$ be a Banach space and $Z$ be a weak$^*$ closed subspace of $Y^*$. It is then standard that $(Z_\perp)^*$ is isometric to $Y^*/(Z_\perp)^\perp=Y^*/Z$ (see e.g. \cite[Theorems 4.7 and 4.9]{Rudin}). Applying this to the situation where $X$ is a Banach space, $Y\subset X^*$ is a predual of $X$, and $Z$ is a $\sigma(X,Y)$-closed subspace of $X$, we get a surjective isometry $S:X/Z\to (Y\cap Z^\perp)^*$ given by $\duality{S([x]),y}=\duality{Tx,y}=\duality{y,x}$ for $x\in X$, $y\in Y\cap Z^\perp$ (here $[x]=x+Z$ denotes the element of $X/Z$ containing $x$, and $T$ is as in the first paragraph of the introduction).

In particular, under hypothesis (i) this yields an isometry $S:\Lip_0(M)/\ideal{N}\to (Y\cap\ideal{N}^\perp)^*$.
On the other hand, by \cite[Lemma 2.27]{Weaver2}, $\Lip_0(N)$ is always isometric to $\Lip_0(M)/\ideal{N}$, with an isometry $R$ given by $f\mapsto [F]$ where $F$ is any Lipschitz extension of $f$ to $M$, whose existence is guaranteed by McShane's theorem. Combining both, we get a surjective isometry
$$
Q=SR:\Lip_0(N)\to (Y\cap\ideal{N}^\perp)^*
$$
given by $\duality{Qf,\phi}=\duality{\phi,F}$ for $f\in\Lip_0(N)$ and $\phi\in Y\cap\ideal{N}^\perp$, where $F$ is any Lipschitz extension of $f$ to $M$.

By hypothesis (ii), $Q$ must be weak$^*$-to-weak$^*$ continuous, so there exists a surjective isometry $P:Y\cap\ideal{N}^\perp\to\lipfree{N}$ such that $Q=P^*$.
Now let $x\in N$. Then there exists $\phi\in Y\cap\ideal{N}^\perp$ such that $P\phi=\delta(x)\in\lipfree{N}$. For any $f\in\Lip_0(M)$, we have that $f$ is an extension of $f\restrict_N$ and therefore
$$
\duality{\phi,f} = \duality{Q(f\restrict_N),\phi} = \duality{f\restrict_N,P\phi} = \duality{f\restrict_N,\delta(x)} = f(x) .
$$
It follows that $\phi=\delta(x)\in\lipfree{M}$. Thus $Y$ contains all $\delta(x)$, $x\in N$, and therefore also $\mathcal{F}_M(N)$.
\end{proof}

We will now adapt the first step in Weaver's argument to make it work in a finite-dimensional setting. The reduction to the bounded case will be replaced with a dimension reduction argument that makes it possible to work by induction on the dimension, ultimately relying on the $1$-dimensional case which is well-known.

We start by modifying Weaver's proof slightly so that the result holds for any ball, not necessarily one containing the base point, at the cost of replacing $\ideal{B}$ with $\const{B}$.

\begin{proposition}
\label{pr:k ball}
Let $M$ be a geodesic metric space. If $B$ is a closed ball in $M$ then $\const{B}$ is weak$^*$ closed with respect to any predual of $\Lip_0(M)$.
\end{proposition}

\begin{proof}
Suppose that $B$ has center $p\in M$ and radius $r>0$. Define the following functions for $x\in M$:
$$
g(x) = \min\set{r,d(x,p)} \qquad\text{and}\qquad h(x) = g(x)-g(0) .
$$
Then $h\in B_{\Lip_0(M)}$, $h(x)-h(p)=d(x,p)$ for $x\in B$, and $h$ is constant outside of $B$. We claim that
\begin{equation}
\label{eq:intersection of balls}
B_{\const{B}} = B_{\Lip_0(M)} \cap (B_{\Lip_0(M)}+h) \cap (B_{\Lip_0(M)}-h) .
\end{equation}
This will be enough to prove the proposition. Indeed, $B_{\Lip_0(M)}$ and its translates are weak$^*$ compact with respect to any predual, so the same holds for $B_{\const{B}}$, and the Krein-\v{S}mulyan theorem implies that $\const{B}$ is weak$^*$ closed.

We first prove inclusion $\supset$ in \eqref{eq:intersection of balls}. Suppose that $f\in B_{\Lip_0(M)}$ is such that $\lipnorm{f\pm h}\leq 1$. Let $x\in B$, then
$$
f(x)-f(p) = (f+h)(x) - (f+h)(p) - (h(x)-h(p)) \leq d(x,p) - d(x,p) = 0
$$
and
$$
f(x)-f(p) = (f-h)(x) - (f-h)(p) + (h(x)-h(p)) \geq -d(x,p) + d(x,p) = 0
$$
so $f(x)=f(p)$. Since $x$ was arbitrary, we get $f\in\const{B}$ as needed.

Now we prove $\subset$ in \eqref{eq:intersection of balls}. Suppose that $f\in B_{\Lip_0(M)}$ is constant in $B$. We wish to prove that $\lipnorm{f\pm h}\leq 1$. Let $x,y\in M$. If $x,y\in B$ then
$$
\abs{(f\pm h)(x)-(f\pm h)(y)} = \abs{h(x)-h(y)} \leq d(x,y)
$$
as $f$ is constant on $B$. Similarly, if $x,y\in M\setminus B$ then
$$
\abs{(f\pm h)(x)-(f\pm h)(y)} = \abs{f(x)-f(y)} \leq d(x,y)
$$
as $h$ is constant outside of $B$. Finally, suppose $x\in B$ and $y\notin B$. Since $M$ is geodesic, there exists $z\in M$ such that $d(x,z)+d(z,y)=d(x,y)$ and $d(z,p)=r$. Then we have
\begin{align*}
\abs{(f\pm h)(x)-(f\pm h)(y)} &\leq \abs{(f\pm h)(x)-(f\pm h)(z)} + \abs{(f\pm h)(z)-(f\pm h)(y)} \\
&= \abs{h(x)-h(z)} + \abs{f(z)-f(y)} \\
&\leq d(x,z) + d(z,y) = d(x,y)
\end{align*}
as $f(z)=f(x)$ and $h(z)=h(y)$. This ends the proof.
\end{proof}

Now that the restriction on containing $0$ has been lifted, the result can be extended to more general domains.

\begin{proposition}
\label{pr:k open connected}
Let $M$ be a geodesic metric space. If $U\subset M$ is open and connected then $\const{U}$ is weak$^*$ closed with respect to any predual of $\Lip_0(M)$.
\end{proposition}

\begin{proof}
This follows from Proposition \ref{pr:k ball} by observing that
$$
\const{U} = \bigcap\set{\const{B} \,:\, \text{$B$ is a closed ball contained in $U$} } .
$$
Indeed, inclusion $\subset$ is obvious. For the converse, note that any $f\in\bigcap_B\const{B}$ is locally constant at every point of $U$ because $U$ is open. Thus level sets $f^{-1}(c)\cap U$ are simultaneously open and closed in $U$ and, since $U$ is connected, $f$ must be constant in $U$.
\end{proof}

Next, we extend the result to a family of subsets of $M$ modeled after hyperplanes in Banach spaces. We give them a name for simplicity.

\begin{definition}
\label{def:wall}
A non-empty closed subset $W\subset M$ is a \emph{geodesic wall} in $M$ if there exist two sets $A,B\subset M$ such that
\begin{enumerate}[label={\upshape{(\alph*)}}]
\item $M$ equals the disjoint union $A\cup W\cup B$,
\item $W\subset\cl{A}\cap\cl{B}$,
\item $A,B$ are open and connected, and
\item for any $a\in A$, $b\in B$ there exists $w\in W$ such that $d(a,w)+d(w,b)=d(a,b)$.
\end{enumerate}
\end{definition}

\begin{proposition}
\label{pr:k wall}
Let $M$ be a geodesic metric space. If $W$ is a geodesic wall in $M$, then $\const{W}$ is weak$^*$ closed with respect to any predual of $\Lip_0(M)$.
\end{proposition}

\begin{proof}
Let $A,B$ be as in Definition \ref{def:wall}. We shall show that
\begin{equation}
\label{eq:sum of balls}
B_{\const{W}} = B_{\const{A}} + B_{\const{B}} .
\end{equation}
Again, this will be enough as $B_{\const{A}}$ and $B_{\const{B}}$ are weak$^*$ compact with respect to any predual by Proposition \ref{pr:k open connected}, so the same holds for their sum, thus $\const{W}$ is weak$^*$ closed by the Krein-\v{S}mulyan theorem.

Since $\cl{A}\cup\cl{B}=M$, we assume without loss of generality that the base point $0$ belongs to $\cl{A}$. Given $f\in\const{W}$ with $\lipnorm{f}\leq 1$, let $c$ be the constant value of $f$ at $W$ and define functions $g$ and $h$ on $M$ by
$$
g(x) = \begin{cases}
0 &\text{, if $x\in A$} \\
f(x)-c &\text{, if $x\notin A$}
\end{cases}
\qquad\text{and}\qquad
h(x) = \begin{cases}
c &\text{, if $x\in B$} \\
f(x) &\text{, if $x\notin B$}
\end{cases}
.
$$
It is clear that $g$ is constant on $\cl{A}\supset A\cup W$, $h$ is constant on $\cl{B}\supset B\cup W$, $g(0)=h(0)=0$, and $f=g+h$. We may also check that $\lipnorm{g}\leq 1$. Indeed, we have $g(x)-g(y)=0$ if $x,y\in\cl{A}$ and $\abs{g(x)-g(y)}=\abs{f(x)-f(y)}\leq d(x,y)$ if $x,y\in\cl{B}$. Now assume $x\in A$ and $y\in B$. Then there exists $w\in W$ such that $d(x,w)+d(w,y)=d(x,y)$, and
$$
\abs{g(x)-g(y)} = \abs{f(y)-c} = \abs{f(y)-f(w)} \leq d(w,y) \leq d(x,y) .
$$
Similar reasoning yields $\lipnorm{h}\leq 1$. Hence $g\in B_{\const{A}}$ and $h\in B_{\const{B}}$. This establishes inclusion $\subset$ in \eqref{eq:sum of balls}.

Conversely, let $g\in B_{\const{A}}$ and $h\in B_{\const{B}}$, and set $f=g+h$. Clearly $f$ is constant on $W\subset\cl{A}\cap\cl{B}$. Again, we have $\abs{f(x)-f(y)}=\abs{h(x)-h(y)}\leq d(x,y)$ if $x,y\in\cl{A}$ and $\abs{f(x)-f(y)}=\abs{g(x)-g(y)}\leq d(x,y)$ if $x,y\in\cl{B}$. For $x\in A$ and $y\in B$, choose $w\in W$ such that $d(x,w)+d(w,y)=d(x,y)$, then
\begin{align*}
\abs{f(x)-f(y)} &\leq \abs{f(x)-f(w)} + \abs{f(w)-f(y)} \\
&= \abs{h(x)-h(w)} + \abs{g(w)-g(y)} \\
&\leq d(x,w) + d(w,y) = d(x,y) .
\end{align*}
We conclude that $\lipnorm{f}\leq 1$ and $f\in B_{\const{W}}$. This completes the proof of \eqref{eq:sum of balls}.
\end{proof}

We may now finally prove Theorem \ref{th:theorem 2} by applying Proposition \ref{pr:k wall} to hyperplanes of increasing dimension.

\begin{proof}[Proof of Theorem \ref{th:theorem 2}]
The theorem is obviously true if $M$ is a singleton. Otherwise, we may assume without loss of generality that $M$ is a convex subset of an $n$-dimensional Banach space $X$ whose interior contains the origin $0$ of $X$, which we take as base point.

We proceed by induction on the dimension $n$. If $n=1$ then $M$ is an interval, possibly of infinite length. In that case, it is well known that $\Lip_0(M)$ is linearly isometric to $L_\infty(M)$ via the mapping $f\mapsto f'$ (see e.g. \cite[Example 2.7]{Weaver2}), so it has a strongly unique predual.

Now assume that the theorem holds for dimension $n-1$. Let $H$ be any hyperplane of $X$ containing $0$. Then $H\cap M$ is a convex subset of $H$, which has dimension $n-1$, so $\Lip_0(H\cap M)$ has a strongly unique predual by induction hypothesis. Note that $H\cap M$ is a geodesic wall in $M$: in Definition \ref{def:wall} we may take $A$ and $B$ to be the intersection with $M$ of the two open half-spaces determined by $H$, which are open in $M$ and convex, hence connected, and the point $w$ of item (d) to be the intersection of the linear segment $[a,b]$ with $H$, which belongs to $M$ by convexity. Since $H\cap M$ contains $0$, Proposition \ref{pr:k wall} implies that $\ideal{H\cap M}=\const{H\cap M}$ is weak$^*$ closed with respect to any predual of $\Lip_0(M)$.

Let $Y\subset\Lip_0(M)^*$ be any predual of $\Lip_0(M)$. By Lemma \ref{lm:predual and ideal} we get that $\lipfree{H\cap M}\subset Y$. Since this is true for every hyperplane $H$, and every non-zero $x\in M$ is contained in some such hyperplane by the Hahn-Banach theorem, we get $\delta(x)\in Y$ for all $x\in M$ and thus $\lipfree{M}\subset Y$. But $\lipfree{M}$ is already a predual of $\Lip_0(M)$, so we conclude $Y=\lipfree{M}$ as it is not possible for one predual to strictly contain another one by the same argument used at the end of the proof of Proposition \ref{pr:godefroy criterion}. Thus $\lipfree{M}$ is the strongly unique predual of $\Lip_0(M)$.
\end{proof}

\begin{remark}
Propositions \ref{pr:k ball} and \ref{pr:k open connected} remain valid if one assumes that $M$ is length rather than geodesic. Recall that a metric space $M$ is a \emph{length space} if, given any $x,y\in M$ and any $\ep>0$, there is a rectifiable path in $M$ joining $x$ and $y$ whose length is less than $d(x,y)+\ep$. In that case, the last part of the proof of Proposition \ref{pr:k ball} can be modified as follows: for every $\ep>0$, one may find $z\in M$ such that $d(x,z)+d(z,y)\leq d(x,y)+\ep$ and $d(z,p)=r$. Then the computation yields $\abs{(f\pm h)(x)-(f\pm h)(y)}\leq d(x,y)+\ep$ and, since $\ep$ was arbitrary, one concludes $\lipnorm{f\pm h}\leq 1$ anyway.

A similar adjustment shows that Proposition \ref{pr:k wall} stays valid when $M$ is length and $W$ is a \emph{length wall}, where condition (d) in Definition \ref{def:wall} is weakened to: \textit{for any $a\in A$, $b\in B$ and $\ep>0$ there exists $w\in W$ such that $d(a,w)+d(w,b)\leq d(a,b)+\ep$.}

These stronger versions are not relevant to our finite-dimensional setting because length spaces that are complete and locally compact are already geodesic by the Hopf-Rinow theorem (see e.g. \cite[Proposition 3.7]{BridsonHaefliger}).
\end{remark}

\section*{Acknowledgments}

F. Vico was supported by the Simons Foundation under project SFI-FI-CCM-Grant-00030182.


\section*{AI disclosure statement}

We acknowledge the use of ChatGPT 5.6 Pro by OpenAI as a tool to obtain the results in this note and to double-check the validity of the final text. The manuscript was entirely written by the authors, who verified its content independently and take full responsibility for it.


\end{document}